\documentclass[11pt]{article}
\usepackage{amsmath, amssymb, amsthm}
\usepackage{bbm}
\usepackage[nottoc]{tocbibind}
\usepackage{enumitem}
\usepackage{hyperref}
\usepackage{tikz}
\usetikzlibrary{calc}
\usepackage{xcolor}
\usepackage{multirow}
\usepackage{comment}

\newtheorem{theorem}{Theorem}[section]
\newtheorem{lemma}{Lemma}[section]
\newtheorem{conj}{Conjecture}[section]
\newtheorem{definition}{Definition}[section]
\newtheorem{corollary}{Corollary}[section]
\newtheorem{proposition}{Proposition}[section]
\newtheorem{remark}{Remark}[section]

\newcommand{\N}{\mathbb{N}}

\newcommand{\Z}{\mathbb{Z}}

\newcommand{\prob}{\mathbb{P}}

\newcommand{\stoch}{\text{stoch}}
\newcommand{\stat}[2]{\mu^{\stoch}_{#1,#2}}
\newcommand{\statn}[1]{\mu^{\stoch}_{#1}}
\newcommand{\one}{\mathbbm{1}}

\newcommand{\sgrleft}[2]{\mathsf{l}_{#1}(#2)}
\newcommand{\sgrright}[2]{\mathsf{r}_{#1}(#2)}
\newcommand{\sgrboth}[2]{\mathsf{b}_{#1}(#2)}
\newcommand{\trans}[1]{P_{#1}}

\usepackage[left=1in, right=1in, bottom = 1in, top = 1in]{geometry}
\title{Stochastic Sandpiles with Uniform Toppling Rule on the Line}
\author{David Beck-Tiefenbach, Robin Kaiser}
\date{\today}

\begin{document}
\maketitle
\begin{abstract}
    We consider the stochastic sandpile model with uniform toppling rule on the integer line. During a uniform toppling, with probability $1/3$ one particle is sent to the right of the toppled vertex, with probability $1/3$ one particle is sent to the left, and with probability $1/3$ two particles are sent out, one to the right and one to the left. We calculate exactly the stationary distribution of the stochastic sandpile Markov chain with this toppling rule on finite, connected subsets of the integers, and show that the infinite volume limit exists and is equal to the Dirac measure of the full configuration. For this end, we analyze where the excess mass leaves the system, when stabilizing the full configuration plus one additional particle on finite, connected subsets of the integers.
\end{abstract}

{\it Keywords:} stochastic sandpiles, abelian, stationary distribution, infinite volume limit, uniform toppling, sandpile gambler's ruin, hole probabilities, line.

{2020 Mathematics Subject Classification.} 60J10, 60K35.

\section{Introduction}\label{sec:intro}
In 1987, Per Bak, Chao Tang and Kurt Wiesenfeld introduced the abelian sandpile model in \cite{BTW87,BTW88}. Originally known under the name Bak-Tang-Wiesenfeld model, it was the first discovered dynamical system that reached a critical state on its own, without further fine-tuning of external parameters; a concept known as self-organized criticality. Deepak Dhar later analyzed the model further, and showed that the toppling dynamics are abelian \cite{D90}, which led him to coin the term {\it abelian sandpile model}.

In the abelian sandpile model, the toppling operations are deterministic, where upon toppling of a vertex particles are sent out along every adjacent edge to the neighbouring vertices. The only source of randomness thus comes from the choice of the initial configuration and the placement of the particles. Therefore, a natural generalization of the model is to allow the toppling operations themselves to be random outcomes. That is, upon toppling a vertex, particles are distributed at random among the neighbours of the toppled vertex. This generalization is known as the {\it stochastic sandpile model}.

In this manuscript, we want to analyze the stochastic sandpile model on the integer line, where the toppling operations are given as the {\it uniform topplings}. During a uniform toppling, every possibility of sending out one particle along every adjacent edge is equally likely. That is, on the line, when toppling a vertex, with probability $1/3$ a single particle is either sent to the right or the left, and with probability $1/3$ two particles are sent out, one to the right and one to the left.

More specifically, we analyze the {\it stochastic sandpile Markov chain} on finite line segments. The stochastic sandpile Markov chain evolves by adding a particle at a uniformly chosen random location, and then stabilizing the resulting sandpile configuration. For the deterministic counterpart, the abelian sandpile Markov chain, it is known that the stationary distribution will always be uniform on the recurrent sandpile configurations on any finite graph with a sink \cite{D90,HLMP08}. Furthermore, for the abelian sandpile Markov chain, we can describe the eigenvectors via the so-called {\it multiplicative harmonic functions} \cite{JLP19}, which can be used to prove the cut-off phenomenon on the two-dimensional integer grid \cite{HJL19}. It has even been shown that the infinite volume limit exists on $\Z^d$ \cite{AJ04}, and a sufficent condition for the existence of the infinite volume limit has been found for arbitrary graphs in \cite{J12,JW14}. Also the density of the infinite volume limit has been calculated on $\Z^2$ \cite{P94,P10,P11,KW15,KW16}, as well as the Sierpi\'nski gasket graph \cite{HKS25}. We refer the reader to \cite{J18} for an introduction to the abelian sandpile model and a collection of further results on the model.

For the stochastic sandpile Markov chain, on the other hand, results concerning the stationary distribution or its infinite volume limit are more meagre. In \cite{ptoppling-CMS}, a combinatorial description of the recurrent sandpiles is given for the $p$-toppling rule, which is similar to the uniform toppling rule of the current manuscript. In \cite{CF25} an exact sampling algorithm for the stationary distribution is described, yet again for a different toppling rule, where on all graphs the toppling threshold is $2$.

For the uniform toppling rule on finite line segments, we are able to fully describe the set of recurrent sandpiles, as well as the stationary distribution. We are also able to show the existence of the infinite volume limit on the integers. Our main results are described in the following theorem.

\begin{theorem}\label{thm:main}
    Let for $a,b\in\Z$ with $a<0<b$ denote by $\stat{a}{b}$ the stationary distribution of the stochastic sandpile Markov chain with uniform toppling rule on the set $[[a,b]]:=\{a,a+1,...,b-1,b\}$, where the outside is identified as the sink.
    
    \begin{enumerate}[label=(USS\arabic*),leftmargin=*]
        \item It holds that
        $$\stat{a}{b}(\mathbbm{1}_{[[a,b]]})=\frac{1}{2}+\frac{1}{2(|a|+b+2)},$$
        where $\one_{[[a,b]]}$ denotes the full configuration with one particle at every vertex of $[[a,b]]$.\label{USS1}

        \item The stationary distribution $\stat{a}{b}$ is uniform on the vacant configurations. That is, for all $j\in[[a,b]]$ it holds that
        $$\stat{a}{b}(\mathbbm{1}_{[[a,b]]}-\delta_j)=\frac{1}{2(|a|+b+2)},$$
        where $\delta_j$ is the configuration with one particle at $j$, and $0$ everywhere else.\label{USS2}

        \item The infinite volume limit of the stationary distributions on the line exists and is equal to the Dirac measure of the full configuration on $\Z$, i.e.
        $$\lim_{[[a,b]]\uparrow \Z}\stat{a}{b}=\delta_{\mathbbm{1}_\Z}.$$
        Here, the limit denotes the uniform weak limit, that is, for all sequences $(a_n)_{n\in\N}\in\Z_{<0}^\N$ and $(b_n)_{n\in\N}\in\Z_{>0}^\N$ with
        \begin{align*}
            \lim_{n\rightarrow\infty}a_n=-\infty,&&\text{and}&&\lim_{n\rightarrow\infty}b_n=\infty,
        \end{align*}
        and all finite subsets $A\subseteq\Z$ it holds that
        $$\lim_{n\rightarrow\infty}\stat{a_n}{b_n}(\eta|_A=\mathbbm{1}_A)=1.$$\label{USS3}
        \end{enumerate}
\end{theorem}
Theorem \ref{thm:main} shows a remarkable difference between the stochastic sandpile stationary distribution, and its deterministic counterpart, the stationary distribution of the abelian sandpile model. For the abelian sandpile model, it is well-established, that on any finite graph with a sink vertex, the stationary distribution of the corresponding abelian sandpile Markov chain is uniform on all recurrent configurations. Despite the symmetries of the uniform toppling rule, the behaviour of the stochastic sandpile Markov chain is far from uniform, in fact, it puts half of its total weight on the full configuration, and thus seems to favor an abundance of particles. However, restricted to only the vacant configurations missing one particle, the stationary distribution of the stochastic sandpile model still remains uniform.

The proof of Theorem \ref{thm:main} relies on analyzing where the excess mass leaves the system upon stabilizing the full configuration plus one additional particle. In fact, if we know where the excess mass leaves our finite subset, and where the vertices with $0$ particles after stabilization occur, we can fully calculate the transition probabilities of the stochastic sandpile Markov chain. With the transition probabilities in hand, we can then analyze the stationary distribution.

\textbf{Outline.} In Section \ref{sec:prelim} we will introduce all the necessary concepts in order to properly state and define our model, and to prove Theorem \ref{thm:main}. Section \ref{sec:concentration} will be devoted to proving part \ref{USS1} of Theorem \ref{thm:main}, that is we show that the stationary distribution concentrates around half of its total mass on the full configuration. For this end, we will analyze the sandpile gambler's ruin problem in Subsection \ref{sec:gambler}, where we examine where the excess mass leaves the system upon stabilization of the full configuration plus one additional particle. With the results of the sandpile gambler's ruin, we will then prove \ref{USS1} in Subsection \ref{sec:conc-on-full}. In Section \ref{sec:infinite-volume-limit}, we will show that the infinite volume limit exists and is fully supported on the full configuration on $\Z$. Here, we will need to refine our analysis of the sandpile gambler's ruin, and investigate where the vertices with $0$ particles appear when two particles are lost during stabilization of the full configuration plus one additional particle. With these results, we will then show \ref{USS2} in Subsection \ref{sec:vacant}, and finally the infinite volume limit \ref{USS3} will follow in Subsection \ref{sec:proof-infinite-volume}.

\subsection{Related Works}\label{sec:relwork}
In this section we want to give an overview of related works concerning the stochastic sandpile model and related models. The following is by no means a full and detailed exposition of works on stochastic sandpiles, but merely an overview to allow for a proper placement of the current manuscript within the research area.

For the stochastic sandpile model and its cousin, the activated random walk model, we typically differentiate between the {\it driven-dissipative} dynamics and the {\it fixed-energy} dynamics.

In the driven-dissipative dynamics, we consider the model on finite graphs with designated sink, and keep adding and stabilizing the configuration until it reaches an equilibrium state. In contrast, for the fixed-energy dynamics, we add to every vertex of our graph a random number of particles, and then try to stabilize the resulting initial configuration in finite time. A popular conjecture, the {\it density conjecture}, claims, that the stationary density of driven-dissipative dynamics coincides with the critical density of the fixed-energy dynamics.

\textbf{Stochastic Sandpiles.} For the stochastic sandpile model, several random toppling operations, different from the uniform topplings of this manuscript, have been considered before. In \cite{M91}, a model is studied, where instead of a toppling threshold, the system is driven by a hard repulsion of two different types of particles which cannot co-exist at the same location.

In \cite{ptoppling-CMS}, a toppling rule similar to the uniform toppling is introduced. For a fixed parameter $p\in[0,1]$, upon toppling a vertex, for each edge adjacent to the toppled vertex independently of the other edges, a particle is sent out along the edge with probability $p$. \cite{S24} analyzes the combinatioral aspect of this {\it $p$-toppling} on complete graphs, and in \cite{SZ25} a similar analysis is done on complete bipartite graphs. We want to emphasize here, that the dynamics of the uniform toppling rule coincide with $p$-topplings for $p=1/2.$

Yet another toppling rule is investigated in \cite{HHRR22}, where it is shown that the critical density on the integers is strictly below $1$, proving a conjecture by Rolla and Sidoravicius \cite{RS12}. Here, a vertex is declared unstable if it is occupied by more than $2$ particles, upon which a toppling lets two particles perform a random walk step to the neighbours of the toppled vertex. It has been shown in \cite{CFT24}, that the critical density is strictly below $1$ in higher dimensions as well.

\textbf{Acitvated Random Walks.} Activated random walks are a stochastic network, where active particles perform a random walk until they either fall asleep, or reach a designated sink vertex and are removed from the system. Should two particles occupy the same vertex at the same time, both particles become active again and continue performing random walks independently. See \cite{R15,LL24} for a nice overview of the topic and some related research questions.

For activated random walks, it has also been shown that the critical density is strictly less than $1$ on the Euclidean lattice in arbitrary dimensions \cite{AFG24}. In the driven-dissipative dynamics, also bounds on the mixing time have been derived and an exact sampling algorithm for the stationary distribution has been proven \cite{LL24}. Also the density conjecture has been proven to hold true in dimension one \cite{HJJ24}.

Activated random walks are believed to belong to the same universality class as the stochastic sandpile model considered in \cite{HHRR22,CFT24}, which highlights a deep connection between these two models \cite{L04}.
\section{Preliminaries}\label{sec:prelim}
In this section we will define the stochastic sandpile model and collect all results necessary to prove Theorem \ref{thm:main}.

\textbf{Uniform Toppling Rules.} Let $a,b\in\Z$ with $a<b$ and consider the set $[[a,b]]:=\{a,a+1,...,b-1,b\}$. For every $v\in[[a,b]]$ we consider a random variable $L_v$ taking values in $\Z^{[[a,b]]}$ with the following distribution
\begin{align*}
    \prob(L_v=\delta_{v-1}-\delta_v)=\prob(L_v=\delta_{v+1}-\delta_v)=\prob(L_v=\delta_{v+1}+\delta_{v-1}-2\delta_v)=\frac{1}{3}.
\end{align*}
Here, $\delta_i$ for $i\in[[a,b]]$ denotes the Dirac delta function
$$\delta_i:[[a,b]]\rightarrow\Z,\hspace{1cm}j\mapsto\begin{cases}
    1,&j=i\\
    0,&\text{else}
\end{cases},$$
and $\delta_i$ is constantly $0$ if $i\notin[[a,b]].$ We call $L_v$ the {\it uniform toppling} or the {\it uniform toppling rule} at $v$, and we call the collection of toppling rules $(L_v)_{v\in[[a,b]]}$ the {\it uniform toppling rules} on the line segment $[[a,b]]$. Notice that topplings at $a$ and $b$ have the chance to remove a particle from the system, we will thus refer to the outside as the sink vertex of the line segment $[[a,b]].$

\textbf{Stochastic Sandpiles.} Let $a,b\in\Z$ with $a<b$. Furthermore, let for every $v\in[[a,b]]$ an i.i.d. sequence of random toppling instruction $(L_v^i)_{i\in\N}$ be given, where
$$L_v^1\sim L_v.$$
Here "$\sim$" denotes that the two random variables are equal in distribution. A {\it sandpile configuration} on $[[a,b]]$ is a function $\eta:[[a,b]]\rightarrow\Z$, that assigns to every vertex of the line segment the number of particles that is placed at that vertex. Given a vertex $v\in[[a,b]]$, we say that $\eta$ is unstable at $v$, if $\eta(v)\geq 2$. We say the sandpile $\eta$ is unstable if it is unstable at some vertex $v\in V$, and stable otherwise.

Given a sandpile configuration $\eta$, {\it stabilization} of the sandpile is then defined in the following manner. Set $\eta_0:=\eta$ and let $u_0$ be the constant $0$ function on $[[a,b]]$. We call $u_0$ the odometer function. If $\eta$ is already stable at every vertex, then the process immediately terminates.

Otherwise, choose among the unstable vertices one vertex uniformly at random, and denote the choice by $v_1$. Then we set
$$\eta_1=\eta_0+L_{v_1}^{u_0(v_1)+1},\hspace{1cm}u_1=u_0+\delta_{v_1}.$$
In words, we add a realization of the toppling rule at vertex $v_1$ to the sandpile configuration and increment the odometer function by $1$ at $v_1$.

If now for some $n\in\N$ the configuration $\eta_n$ is stable at all vertices, then the process terminates. Otherwise, choose one vertex uniformly at random from all unstable vertices and set
$$\eta_{n+1}=\eta_n+L_{v_{n+1}}^{u_n(v_{n+1})+1},\hspace{1cm}u_{n+1}=u_n+\delta_{v_{n+1}},$$
where $v_{n+1}$ is the chosen vertex. This stabilization will eventually terminate for every initial sandpile configuration $\eta$.
\begin{lemma}[Stabilization is Well-Defined]\label{lemma:stabilization-well-defined}
    For all initial sandpile configurations $\eta_0:[[a,b]]\rightarrow\Z$, there exists almost surely an $N\in\N$ such that $\eta_N$ is stable at all vertices.
\end{lemma}
\begin{proof}
    Assume the contrary, that is there exists some configuration $\eta_0$ which does not stabilize. Since $[[a,b]]$ is finite, there exists a vertex $v\in[[a,b]]$ which topples infinitely often. Since $L_v$ has positive probability of sending a particle to either the left or the right, both $v-1$ and $v+1$ receive an infinite amount of particles from $v$. And because in every step we choose the toppled vertex uniformly at random, also $v-1$ and $v+1$ must topple infinitely often.

    We thus see inductively that all vertices in $[[a,b]]$ must topple infinitely often, which means that also $a$ topples infinitely often. However, since toppling at $a$ has positive probability of reducing the total number of particles in the system by $1$, an infinite amount of particles is lost during stabilization. However, $\eta_0$ starts with only a finite amount of particles, thus we arrive at a contradiction.

    This means, that all $\eta_0$ must stabilize in finitely many steps almost surely.
\end{proof}
Given a sandpile configuration $\eta$, we will denote by $\eta^\circ$ the final stabilized configuration. Notice that in the stochastic sandpile model, $\eta^\circ$ is a random variable, that is the outcome of the stabilization is itself already random, given that the initial configuration was unstable.

\textbf{Abelian Property of Topplings.} An important property of the stochastic sandpile model is, that the order in which we topple our vertices does not influence the final result. Formally, this means that if we have two unstable vertices $v,w\in [[a,b]]$, if we first topple at $v$ and then at $w$ the final configuration will have the same law as if we first toppled at $w$ and then at $v$. This follows directly from the fact that topplings are defined as additions of random columns, i.e.
$$(\eta+L_v^i)+L_w^j\sim (\eta+L_w^j)+L_v^i,$$
for all $v,w\in V$ and $i,j\in\N$ and all sandpile configuration $\eta$.

The abelian nature of the toppling operations allows us to choose the order in which we topple the unstable vertices during stabilization of a given sandpile configuration $\eta$, that is, we do not necessarily need to choose the next toppled vertex uniformly among the currently unstable vertices.

We will repeatedly make use of the abelian property throughout this paper, by first stabilizing a given sandpile configuration in some finite subset $A\subseteq [[a,b]]$, before we stabilize it fully on the whole line segment $[[a,b]].$

\textbf{Stochastic Sandpile Markov Chain.} Using the stabilization dynamics of stochastic sandpiles, we can define for all $a,b\in\Z$ with $a<b$ a Markov chain, which takes values in the set of stable sandpiles on $[[a,b]].$

Let $(X_n^{a,b})_{n\in\N}$ be an i.i.d. sequence of random variables that are uniformly distributed on $[[a,b]].$ Let further $\eta_0$ be some initial, stable sandpile configuration. We then define the stochastic sandpile Markov chain via
$$\eta_{n+1}=(\eta_n+\delta_{X_n^{a,b}})^\circ,$$
that is, in every time step we add an additional particle at a uniformly chosen location, and then stabilize the configuration. We call this Markov chain the {\it stochastic sandpile Markov chain} on $[[a,b]]$. It is not hard to see, that the set of recurrent sandpiles is given as
$$\mathcal{R}_{a,b}:=\{\one_{[[a,b]]}\}\cup\{\one_{[[a,b]]}-\delta_j:j\in[[a,b]]\},$$
where $\one_{[[a,b]]}$ denotes the full configuration on $[[a,b]]$ that has one particle at every vertex. We will further denote by $\stat{a}{b}$ the stationary distribution of this Markov chain supported on $\mathcal{R}_{a,b}.$

\textbf{Infinite Volume Limit.} By the {\it infinte volume limit} of stationary measures we mean the uniform weak limit of the stationary measures on the finite line segments $[[a,b]]$. That is, we say that the infinite volume limit exists and is equal to some measure $\mu$ supported on the set
$\{\eta:\Z\rightarrow \{0,1\}\},$
if for all sequencese $(a_n)_{n\in\N}$ and $(b_n)_{n\in\N}$ with
\begin{align*}
    \forall n\in\N: a_n<0<b_n,&&\text{and}&&\lim_{n\rightarrow\infty}a_n=-\infty,&&\text{and}&&\lim_{n\rightarrow\infty}b_n=\infty,
\end{align*}
it holds for all finite subsets $A\subseteq \Z$ and all configurations $s:A\rightarrow\{0,1\}$ that
$$\lim_{n\rightarrow\infty}\stat{a_n}{b_n}(\eta|_A=s)=\mu(\eta|_A=s).$$
That is, all finite marginals of the stationary measures must converge to the finite marginals of $\mu$, independently of our choice for $(a_n)_{n\in\N}$ and $(b_n)_{n\in\N}.$ We will denote this limit by
$$\lim_{[[a,b]]\uparrow\Z}\stat{a}{b}=\mu.$$
\section{Concentration Phenomenon of the Stationary Distribution}\label{sec:concentration}
In this section we will show that the stationary distribution of the stochastic sandpile Markov chain on finite, connected subsets of the integers concentrates almost half its total weight on the full configuration. For ease of readability, we will throughout this section only work on the sets of the form $C_n:=\{1,...,n\}$ for $n\in\N$. As any finite, connected subset of $\Z$ can be obtained from simply shifting a set $C_n$ of suitable length, all the results of this section will carry over to arbitrary finite, connected subsets of the integers. We will then denote the stationary distribution of the stochastic sandpile Markov chain on $C_n$ with uniform toppling rule by $\statn{n}$.

The idea of this section will be to first analyze what happens during stabilization of the configuration $(\one_{C_n}+\delta_j)$ for $j\in C_n$. More precisely, we want to find out the probabilities of losing exactly one particle during stabilization - either through the left side or the right side - or of losing two particles, necessarily then through both sides.

With these probabilities, we can then find the transition probabilities of the stochastic sandpile Markov chain to the full configuration, which we can then utilize to prove the concentration phenomenon \ref{USS1}.

\subsection{Sandpile Gambler's Ruin}\label{sec:gambler}
In order to analyze the transition probabilities to the full configuration $\mathbbm{1}_{C_n}$, we must precisely track where the excess mass exits the system when stabilizing the configuration $\big(\mathbbm{1}_{C_n} + \delta_j\big)$ for some $j \in C_n$. 

Recall from Section \ref{sec:prelim} that the toppling rules $(L_v)_{v\in C_n}$ are defined such that particles sent outside of $C_n$ are removed. To rigorously count these lost particles, let us define for each toppling $L_v^i$ the indicator variables $\Delta_L(L_v^i) \in \{0,1\}$ and $\Delta_R(L_v^i) \in \{0,1\}$, which equal $1$ if the toppling operation sends a particle to $v-1$ or $v+1$, respectively.

Given an initial configuration $\eta_0$, let $u : C_n \to \mathbb{N}_0$ denote the final odometer function upon stabilization. The total number of particles that exit the system through the left boundary ($0$) and the right boundary ($n+1$) are given by the random variables:
\begin{align*}
    N_L = \sum_{i=1}^{u(1)} \Delta_L(L_1^i), &&
    N_R = \sum_{i=1}^{u(n)} \Delta_R(L_n^i).
\end{align*}

We are interested in the event where exactly one particle is lost during stabilization, stabilizing to the full configuration $\mathbbm{1}_{C_n}$.

\begin{definition}[Sandpile Gambler's Ruin Probabilities]
For $n \in \mathbb{N}$ and $k \in C_n$, we define the probability of exactly one particle exiting through the left boundary upon stabilizing $\mathbbm{1}_{C_n} + \delta_k$ as:
\begin{equation*}
    \sgrleft{n}{k} := \prob\Big( (\mathbbm{1}_{C_n} + \delta_k)^\circ = \mathbbm{1}_{C_n}, \, N_L = 1, \, N_R = 0 \Big).
\end{equation*}
Analogously, we define $\sgrright{n}{k}$ as the probability that exactly one particle exits through the right boundary ($N_L = 0, N_R = 1$), and $\sgrboth{n}{k}$ as the probability that exactly two particles exit the system ($N_L = 1, N_R = 1$).
\end{definition}

For completeness, we extend the domain of these functions to the boundaries by setting $\sgrleft{n}{0} = 1$, $\sgrleft{n}{n+1} = 0$, and similarly $\sgrright{n}{0} = 0$, $\sgrright{n}{n+1} = 1$ as well as $\sgrboth{n}{0}=\sgrboth{n}{n+1}=0.$
\begin{lemma}[Recurrence Relation]\label{lemma:sgr-recurrence}
    For $n \in \N$ and $k \in C_n$, the probabilities $(\sgrleft{n}{k})_{n\in\N,k\in C_n}$ satisfy the recurrence relation
    \begin{equation*}\label{eq:sgr-recurrence}
        \sgrleft{n}{k} = \frac{1}{3}\bigg(\sgrleft{n}{k-1} + \sgrleft{n}{k+1} + \sgrleft{n-k}{1}\sgrleft{n}{k-1}\bigg),
    \end{equation*}
    with boundary conditions $\sgrleft{n}{0} = 1$ and $\sgrleft{n}{n+1} = 0$.
\end{lemma}

\begin{proof}
    Consider the first toppling of the unstable vertex $k$ in the configuration $\eta_0 = \one_{C_n} + \delta_k$.
    With probability $1/3$, the particle is sent to $k-1$, resulting in the configuration $\one_{C_n} + \delta_{k-1}$. The probability of exactly one particle subsequently exiting left is, by definition, $\sgrleft{n}{k-1}$.
    Similarly, with probability $1/3$, the particle is sent to $k+1$, yielding $\one_{C_n} + \delta_{k+1}$ and a left-exit probability of $\sgrleft{n}{k+1}$.
    
    With probability $1/3$, the split toppling occurs, yielding the configuration $\eta_1 = \one_{C_n} - \delta_k + \delta_{k-1} + \delta_{k+1}$. By the Abelian property of the stochastic sandpile model, the final stabilized configuration is independent of the order of subsequent topplings. We may therefore choose to stabilize the sub-segment $\{k+1, \dots, n\}$ first. 
    In order for the global condition $N_R = 0$ to hold, this sub-stabilization must not eject any particles to the right boundary $n+1$. Thus, the single excess particle at $k+1$ must exit the sub-segment to its left, falling into the hole at $k$. This independent sub-stabilization behaves identically to a system of size $n-k$ with a particle added at relative position $1$. This event occurs with probability $\sgrleft{n-k}{1}$.
    
    Conditional on this event occurring, the hole at $k$ is filled, and the right sub-segment is fully stabilized at $\one_{\{k+1, \dots, n\}}$. The remaining unstable configuration on $C_n$ is exactly $\one_{C_n} + \delta_{k-1}$. The probability that the stabilization of this remaining configuration results in a single left exit is $\sgrleft{n}{k-1}$. The joint probability of this split sequence is therefore $\sgrleft{n-k}{1}\sgrleft{n}{k-1}$. Summing the probabilities of these three mutually exclusive initial toppling outcomes yields the recurrence.
\end{proof}

\begin{proposition}[Explicit Solution of the Sandpile Gambler's Ruin]\label{prop:sgr}
    For all $n \in \N$ and $k \in C_n$, the probabilities of exactly one particle exiting the left or right boundary, or exactly two particles exiting, are given respectively by:
    \begin{align*}
        \sgrleft{n}{k} = \frac{(n-k+1)(n-k+2)}{(n+1)(n+2)}, &&
        \sgrright{n}{k} = \frac{k(k+1)}{(n+1)(n+2)}, &&
        \sgrboth{n}{k} = \frac{2k(n-k+1)}{(n+1)(n+2)}.
    \end{align*}
\end{proposition}
\begin{proof}
    We prove this by defining the explicit sequence $$g_n(k) := \frac{(n-k+1)(n-k+2)}{(n+1)(n+2)},$$ and showing that it uniquely satisfies the recurrence relation established in Lemma \ref{lemma:sgr-recurrence}.

    We first show that the recurrence relation uniquely determines the $\sgrleft{n}{k}$ by induction on the system size $n$. For the base case $n=1$, the value $\sgrleft{1}{1} = \frac{1}{3}$, is determined by the boundary conditions. Assume that the probabilities are uniquely determined for all system sizes $m < n$. For system size $n$, the recurrence equation from Lemma \ref{lemma:sgr-recurrence} defines a tridiagonal linear system for the unknown variables $(\sgrleft{n}{k})_{k\in C_n}$, where the coefficients depend on the uniquely determined value $\sgrleft{n-k}{1}$. Because these coefficients are sub-stochastic, the underlying matrix is irreducibly diagonally dominant. Therefore, the linear system is non-singular, guaranteeing that the solution for size $n$ is unique. Thus, any sequence satisfying the recurrence and boundary conditions must be the true probability $(\sgrleft{n}{k})_{k\in C_n}$.

    We will now verify that our proposed function $g_n$ satisfies the recurrence. The boundary conditions are trivially satisfied: $$g_n(0) = \frac{(n+1)(n+2)}{(n+1)(n+2)} = 1 \quad \text{ and } \quad g_n(n+1) = \frac{0}{(n+1)(n+2)} = 0.$$ 
    To show that the recurrence itself is satisfied, we substitute the function $g_n$ into the right hand side of the recurrence equation from Lemma \ref{lemma:sgr-recurrence}. For brevity set $m = n-k$. By the definition of our proposed sequence, the split-term evaluates to $$g_m(1) = \frac{(m-1+1)(m-1+2)}{(m+1)(m+2)} = \frac{m(m+1)}{(m+1)(m+2)} = \frac{m}{m+2}.$$ 
    Substituting the function $g$ into the recurrence equation of Lemma \ref{lemma:sgr-recurrence} now yields
    \begin{align*}
        \frac{1}{3} &\Big( g_n(k-1) + g_n(k+1) + g_m(1)g_n(k-1) \Big) \\
        &= \frac{1}{3(n+1)(n+2)} \bigg[ (m+2)(m+3) + m(m+1) + \left(\frac{m}{m+2}\right)(m+2)(m+3) \bigg] \\
        &= \frac{1}{3(n+1)(n+2)} \bigg[ (m^2 + 5m + 6) + (m^2 + m) + (m^2 + 3m) \bigg] \\
        &= \frac{1}{3(n+1)(n+2)} \bigg[ 3m^2 + 9m + 6 \bigg] \\
        &= \frac{3(m+1)(m+2)}{3(n+1)(n+2)} = g_n(k),
    \end{align*}
    as desired. Since $g_n(k)$ satisfies the recurrence and the solution to the recurrence equation must be unique, we obtain $\sgrleft{n}{k} = g_n(k)$ for all $k\in C_n$ as well. This finishes the inductive step and thus proves the closed-form expression of the probabilities $(\sgrleft{n}{k})_{n\in\N,k\in C_n}$.

    By symmetry of the uniform toppling rules, the probability of a right exit is given by $\sgrright{n}{k} = \sgrleft{n}{n-k+1}$. Substituting this into our proven formula immediately verifies the closed-form expression of right-exits. 

    Finally, observe that the probability of losing exactly two particles is $\sgrboth{n}{k} = 1 - \sgrleft{n}{k} - \sgrright{n}{k}$ for all $n\in\N$ and $k\in C_n$. Substituting the explicit expressions for $\sgrleft{n}{k}$ and $\sgrright{n}{k}$ gives 
     $$\sgrboth{n}{k} = \frac{2k(n-k+1)}{(n+1)(n+2)},$$
     as desired.
\end{proof}

\begin{remark}[Namesake of the Sandpile Gambler's Ruin]
    We chose to call the problem analyzed in this section the {\bf Sandpile Gambler's Ruin}, not only due to its loose resemblance of the classical gambler's ruin problem, but because the classical gambler's ruin problem can be realized as an analogue statement for a different toppling rule.

    Consider for $n\in\N$ and every $i\in C_n$ the random toppling distribution $L^{\text{SRW}}_i$ given by
    \begin{align*}
        \prob(L^{\text{SRW}}_i=\delta_{i-1})=\prob(L^{\text{SRW}}_i=\delta_{i+1})=\frac{1}{2}.
    \end{align*}
    That is, upon toppling any vertex, exactly one particle is sent to the right or left with equal probablity. If we now stabilize the full configuration plus one additional particle with this toppling rule, the excess particle will perform a simple random walk on $C_n$ until it either leaves through the left or right side of $C_n$. Analyzing the probabilities of these two events is known as the \textbf{Classical Gambler's Ruin} problem.
\end{remark}

\subsection{Concentration on the Full Configuration}\label{sec:conc-on-full}
With the results on the sandpile gambler's ruin from Subsection \ref{sec:gambler}, we can now find the transition probabilities to the full configuration. We will do so in the following lemma.
\begin{lemma}[Transition Probabilities to Full Configuration]\label{lemma:transition-to-full}
    Let $n\in\N$ and consider the transition matrix $\trans{n}$ of the stochastic sandpile Markov chain on the set $C_n$. Then:
    \begin{enumerate}
        \item For all $j\in C_n$ it holds that$$\trans{n}(\one_{C_n}-\delta_j,\one_{C_n})=\frac{n+2}{3n}.$$
        \item The full configuration moves to itself with probability given by$$\trans{n}(\one_{C_n},\one_{C_n})=\frac{2}{3}.$$
    \end{enumerate}
\end{lemma}
\begin{proof}
    We will first analyze the transition probabilities of the vacant configurations to the full configuration. So let $n\in\N$ and $j\in C_n$, and consider the transition from $(\one_{C_n}-\delta_j)$ to $\one_{C_n}$. Every step of the stochastic sandpile Markov chain consists of first placing a particle at a uniformly chosen location in $C_n$. If the particle is placed at vertex $j$, then the hole at $j$ is filled and we moved to the full configuration. This event contributes a probability of $1/n$ to the transition probability.

    Secondly, if the additional particle is placed at a vertex $i\in C_n$ with $i<j$, i.e. it is placed to the left side of the $0$, then we can first stabilize the configuration $(\one_{C_n}-\delta_j+\delta_i)$ on the vertices $\{1,...,j-1\}.$ During this stabilization, if we want the final stable configuration to be the full configuration $\one_{C_n}$, exactly one particle must leave the set $\{1,...,j-1\}$, and it must do so through the right side. This corresponds to the sandpile gambler's ruin probability $\sgrright{j-1}{i}$. We thus arrive at a contribution to the transition probability of
    $$\frac{1}{n}\sum_{i=1}^{j-1}\sgrright{j-1}{i}=\frac{1}{n}\sum_{i=1}^{j-1}\frac{i(i+1)}{3j(j+1)}=\frac{(j-1)j(j+1)}{3nj(j+1)}=\frac{j-1}{3n},$$
    where we have used the closed-form expression of $\sgrright{j-1}{i}$ as obtained in Proposition \ref{prop:sgr}.

    Finally, if the particle is placed to the right side of $j$, we obtain through analogue arguments the contribution to the transition probability given by
    $$\frac{1}{n}\sum_{i=1}^{n-j}\sgrleft{n-j}{i}=\frac{1}{n}\sum_{i=1}^{j-1}\frac{(n-j-i+1)(n-j-i+2)}{3(n-j+1)(n-j+2)}=\frac{(n-j)(n-j+1)(n-j+2)}{3n(n-j+1)(n-j+2)}=\frac{n-j}{3n}.$$
    If we now combine all of these three contributions to the transition probability, we obtain
    $$\trans{n}(\one_{C_n}-\delta_j,\one_{C_n})=\frac{1}{n}+\frac{j-1}{3n}+\frac{n-j}{3n}=\frac{n+2}{3n}.$$

    Let us now consider the transition probability of the full configuration to itself. For this event to occur, if a particle is placed at some vertex $i\in C_n$, we must then sent out exactly one particle during the stabilization of the configuration $(\one_{C_n}+\delta_i)$, either through the left or the right. The probability of this event is again given by the sandpile gambler's ruin as
    $$\sgrleft{n}{i}+\sgrright{n}{i}=1-\sgrboth{n}{i}.$$
    Using now again the closed-form expressions obtain in Proposition \ref{prop:sgr} we obtain
    $$\trans{n}(\one_{C_n},\one_{C_n})=\frac{1}{n}\sum_{i=1}^n\bigg(1-\sgrboth{n}{i}\bigg)=1-\frac{1}{n}\sum_{i=1}^n\frac{2i(n-i+1)}{(n+1)(n+2)}=1-\frac{1}{3}=\frac{2}{3}.$$
\end{proof}

With these transition probabilities, we can now prove \ref{USS1} from Theorem \ref{thm:main}. We will do so in the following proposition, which will be phrased on $C_n$.

\begin{proposition}[Stationary Probability of Full Configuration]\label{prop:concentration-on-full}
    For all $n\in\N$ it holds that
    $$\statn{n}(\one_{C_n})=\frac{1}{2}+\frac{1}{2(n+1)}.$$
\end{proposition}
\begin{proof}
    Let $n\in\N$. By the definition of stationary measures we obtain
    \begin{align}\label{eq:stat-of-full}
    \statn{n}(\one_{C_n})=\statn{n}(\one_{C_n})\trans{n}(\one_{C_n},\one_{C_n})+\sum_{i=1}^n\statn{n}(\one_{C_n}-\delta_i)\trans{n}(\one_{C_n}-\delta_i,\one_{C_n}).
    \end{align}
    Using Lemma \ref{lemma:transition-to-full} we have
    \begin{align*}
        \sum_{i=1}^n\statn{n}(\one_{C_n}-\delta_i)\trans{n}(\one_{C_n}&-\delta_i,\one_{C_n})=\sum_{i=1}^n\statn{n}(\one_{C_n}-\delta_i)\frac{n+2}{3n}\\&=\frac{n+2}{3n}\sum_{i=1}^n\statn{n}(\one_{C_n}-\delta_i)=\frac{n+2}{3n}\big(1-\statn{n}({\one_{C_n}})\big).
    \end{align*}
    Using this in Equation (\ref{eq:stat-of-full}) together with the value of $\trans{n}(\one_{C_n},\one_{C_n})$ from Lemma \ref{lemma:transition-to-full} we obtain
    $$\statn{n}(\one_{C_n})=\frac{n+2}{3n}+\frac{n-2}{3n}\statn{n}(\one_{C_n}).$$
    After rearranging this equation we then see
    $$\frac{2n+2}{3n}\statn{n}(\one_{C_n})=\frac{n+2}{3n},$$
    from which we finally arrive at
    $$\statn{n}(\one_{C_n})=\frac{1}{2}+\frac{1}{2(n+1)},$$
    which completes our proof.
\end{proof}
\section{Infinite Volume Limit}\label{sec:infinite-volume-limit}
In this section, we want to show that the infinite volume limit of stationary distributions on finite line segments exists, and that it is given by the Dirac measure of the full configuration on the integers. For this end, we need to get a handle on the stationary probabilties of the vacant configurations, where exactly one vertex has $0$ particles, whereas all the other vertices have $1$ particle.

It turns out, that we can not only obain suitable bounds for the stationary probabilities of the vacant configurations, but we can in fact calculate the stationary distribution on finite line segments exactly. To do so, we need to refine our analysis of the sandpile gambler's ruin from Section \ref{sec:gambler}. So far, we have only found the probabilities of losing exactly one particle, either to the left or to the right, or of losing two particles during stabilization of the full configuration plus one additional particle. In order to gain a grasp on the transition probabilities to vacant configurations - which are needed to obtain the stationary distribution - we however need to find out, where the vertex with $0$ particles is upon losing two particles in the stabilization.

We will take a similar approach as we did in Section \ref{sec:gambler}, that is, we will first find a recurrence equation for the hole probabilities - which are the probabilities for the location of the vertex with $0$ particles - which we will then solve to obtain the desired probabilities. This time, the recurrence equation will be more complicated than the one from \ref{lemma:sgr-recurrence}, however it can still be solved with a closed-form solution.

We will calculate the hole probabilites via a suitable recurrence equation in Section \ref{sec:holes}, which we will then use to calculate the stationary distribution on vacant configurations in Section \ref{sec:vacant}. Finally, we can use the results to prove the infinite volume limit \ref{USS3} in Section \ref{sec:proof-infinite-volume}.

\subsection{Hole Probabilities}\label{sec:holes}

Throughout this section, we will again work on $C_n$ for ease of readability. The results carry over to arbitrary line segments $[[a,b]]$ for $a,b\in\Z$ with $a<b$. When the stabilization of $\one_{C_n} + \delta_i$ loses exactly two particles, the resulting configuration is missing exactly one particle, that is, there will be a single vertex with $0$ particles on it. We will denote this vertex as a {\it hole}. To determine the transition probabilities of the Markov chain, we must precisely track the spatial location of this newly created hole.

\begin{definition}[Hole Probability]
    For $n \in \N$ and $i, j \in C_n$, we define $\mathsf{h}_n^j(i)$ as the probability that stabilizing $\one_{C_n} + \delta_i$ results in a single hole located exactly at $j$:
    \begin{equation*}
        \mathsf{h}_n^j(i) := \prob\Big( (\one_{C_n} + \delta_i)^\circ = \one_{C_n} - \delta_j \Big).
    \end{equation*}
\end{definition}

In order to solve for $\mathsf{h}_n^j(i)$, we establish two distinct recurrence relations by decomposing the domain into sub-segments and applying the abelian property.

\begin{lemma}[Hole Recurrence Relations]\label{lemma:hole-recurrence}
    For $n \ge 2$, the hole probabilities satisfy the following relations. 
    For $i < n$, we have the bulk recurrence:
    \begin{equation*}
        \mathsf{h}_n^j(i) = \sgrright{n-1}{i}\mathsf{h}_n^j(n) + \sum_{x=1}^{j-1} \mathsf{h}_{n-1}^x(i)\mathsf{h}_{n-x}^{j-x}(n-x) + \mathsf{h}_{n-1}^j(i)\sgrright{n-j}{n-j}.
    \end{equation*}
    For $i=n$, we have the boundary recurrence:
    \begin{equation*}
        \mathsf{h}_n^j(n) = 
        \begin{cases}
            \frac{1}{3}\mathsf{h}_n^j(n-1) + \frac{1}{3}\mathsf{h}_{n-1}^j(n-1) & \text{if } j < n, \\
            \frac{1}{3}\mathsf{h}_n^n(n-1) + \frac{1}{3}\sgrleft{n-1}{n-1} & \text{if } j = n.
        \end{cases}
    \end{equation*}
\end{lemma}

\begin{proof}
    We prove the bulk recurrence by first stabilizing the sub-segment $\{1, \dots, n-1\}$.
    If exactly one particle exits this sub-segment to the left, the configuration becomes $\one_{C_n}$ and no hole is formed. If one particle exits to the right, it lands on $n$, yielding the configuration $\one_{C_n} + \delta_n$, which subsequently forms a hole at $j$ with probability $\mathsf{h}_n^j(n)$. Here, the event to send the particle to vertex $n$ has probabilty $\sgrright{n-1}{i}.$
    
    If two particles exit the sub-segment, a hole is formed at some $x \in \{1, \dots, n-1\}$ and the right-exiting particle lands on $n$. This event contributes a probability of $\mathsf{h}_{n-1}^x(i)$. We must now stabilize the configuration on the sub-segment $\{x+1, \dots, n\}$. 
    Crucially, if $x > j$, the hole at $x$ acts as an absorbing boundary. The stabilization of the right segment can emit at most one particle to the left; if it does, it fills the hole at $x$, resulting in the fully occupied configuration $\one_{C_n}$, with potentially a hole at some vertex in the set $\{x+1,\dots, n\}$. Thus, instability cannot propagate leftward past $x$ to create a hole at $j < x$.
    
    If $x < j$, the right segment acts as a system of size $n-x$ and must independently form a hole at relative position $j-x$, giving the sum over $x < j$. If $x = j$, the hole is already in the correct position, so the right segment must simply eject its excess particle to the right without filling $j$, which contributes a probability of $\sgrright{n-j}{n-j}$. Summing these mutually exclusive trajectories yields the bulk recurrence equation.
    
    To prove the boundary recurrence equation, we examine the first toppling at $n$. A right topple exits the system, stabilizing to $\one_{C_n}$, which has no hole. A left topple moves the particle to $n-1$, yielding a remaining probability of $\mathsf{h}_n^j(n-1)$ to form a hole at the desired location $j$. A split toppling creates a hole at $n$ and an excess particle at $n-1$. For $j < n$, stabilizing $\{1, \dots, n-1\}$ must then create a hole at $j$ and eject a particle to the right to fill $n$, which is an event with probability $\mathsf{h}_{n-1}^j(n-1)$. For $j=n$, the sub-segment must not eject a particle right, so it must exit left, yielding the probability $\sgrleft{n-1}{n-1}$.

    If we again sum all of the mutually exclusive probabilities that can occur as explained above, we obtain thus the boundary recurrence equation.
\end{proof}

We are now equipped to state the main result of this subsection: the spatial location of the hole is uniformly distributed across the domain.

\begin{proposition}[Uniformity of the Hole]\label{prop:hole-uniformity}
    For all $n \in \N$ and $i, j \in C_n$, the probability of a hole forming at $j$ is independent of $j$ and is given by:
    \begin{equation*}
        \mathsf{h}_n^j(i) = \frac{\sgrboth{n}{i}}{n} = \frac{2i(n-i+1)}{n(n+1)(n+2)}.
    \end{equation*}
\end{proposition}

\begin{proof}
    We prove this by showing that $g_n^j(i) := \frac{\sgrboth{n}{i}}{n}$ uniquely satisfies the recurrence relations of Lemma \ref{lemma:hole-recurrence}.
    
     We first show that the boundary and bulk recurrence relations uniquely determine the entire sequence of probabilities. We prove this by induction on $n$. The base case $n=1$ is uniquely determined by the toppling rules. Assume the sequence is uniquely determined for all system sizes $m < n$. For system size $n$, evaluating the bulk recurrence at $n-1$ yields $$\mathsf{h}_n^j(n-1) = \sgrright{n-1}{n-1}\mathsf{h}_n^j(n) + K_j,$$ where 
     $$K_j:=\sum_{x=1}^{j-1} \mathsf{h}_{n-1}^x(n-1)\mathsf{h}_{n-x}^{j-x}(n-x) + \mathsf{h}_{n-1}^j(n-1)\sgrright{n-j}{n-j},$$
     depends only on systems of size $m < n$ and the location of the hole $j$, and is therefore fixed. Substituting this into the boundary recurrence yields a linear equation of the form
    $$ \mathsf{h}_n^j(n) = \frac{1}{3} \Big( \sgrright{n-1}{n-1}\mathsf{h}_n^j(n) + K_j \Big) + C_j, $$
    where
    $$j<n:C_j:=\frac{1}{3}\mathsf{h}_{n-1}^j(n-1),\hspace{1cm}C_n:=\frac{1}{3}\sgrleft{n-1}{n-1},$$
    is another fixed constant independent of the hole probabilities of system size $n$. Because the coefficient $\sgrright{n-1}{n-1}/3$ is strictly less than $1$, this equation has a strictly unique scalar solution for $\mathsf{h}_n^j(n)$. With the boundary value uniquely fixed, the bulk recurrence explicitly and uniquely determines the remaining values $\mathsf{h}_n^j(i)$ for $i < n$. Thus, if $g_n^j(i)$ satisfies the recurrences, then $g_n^j(i)= \mathsf{h}_n^j(i)$ for all $n\in\N$ and all $i,j\in C_n$.
    
    We will now verify that our proposed function $(g_n^j(i))_{i,j\in C_n}$ satisfies the boundary recurrence. The left-hand side of the equation simplifies to
    $$g_n^j(n) = \frac{\sgrboth{n}{n}}{n} = \frac{2n}{n(n+1)(n+2)} = \frac{2}{(n+1)(n+2)}.$$
    Now, calculating the right-hand side for $j < n$ gives
    \begin{align*}
        \frac{1}{3} \bigg( g_n^j(n-1) + g_{n-1}^j(n-1)\bigg) &= \frac{1}{3} \bigg(\frac{4(n-1)}{n(n+1)(n+2)} + \frac{2(n-1)}{(n-1)n(n+1)}\bigg)\\ &= \frac{1}{3} \frac{4(n-1) + 2(n+2)}{n(n+1)(n+2)}  = \frac{1}{3} \frac{6n}{n(n+1)(n+2)}  = \frac{2}{(n+1)(n+2)}. 
    \end{align*} 
    Similarly, the right-hand side for $j = n$ gives
    \begin{align*}
        \frac{1}{3} \bigg( g_n^n(n-1) + \sgrleft{n-1}{n-1}\bigg) &= \frac{1}{3} \bigg(\frac{4(n-1)}{n(n+1)(n+2)} + \frac{2}{n(n+1)}\bigg)\\ &= \frac{2}{(n+1)(n+2)}. 
    \end{align*} 
    This shows that $(g_n^j(i))_{i,j\in C_n}$ satisfies the boundary recurrence.
    
    Next, we will verify that our proposed function also satisfies the bulk recurrence. We substitute $g$ into the right-hand side for $i < n$. Notice that by definition, $g_{n-1}^x(i)$ does not depend on $x$. Thus, we can factor it out of the sum $S_j := \sum_{x=1}^{j-1} g_{n-1}^x(i)g_{n-x}^{j-x}(n-x)$ to obtain
    \begin{align*}
        S_j &= \frac{\sgrboth{n-1}{i}}{n-1} \sum_{x=1}^{j-1} \frac{2(n-x)}{(n-x)(n-x+1)(n-x+2)} = \frac{\sgrboth{n-1}{i}}{n-1} \sum_{x=1}^{j-1} \left( \frac{2}{n-x+1} - \frac{2}{n-x+2} \right) \\&= \frac{\sgrboth{n-1}{i}}{n-1} \left( \frac{2}{n-j+2} - \frac{2}{n+1} \right),
    \end{align*}
    Combining $S_j$ with $g_{n-1}^j(i)\sgrright{n-j}{n-j}$, gives
    \begin{align*}
        S_j + g_{n-1}^j(i)\sgrright{n-j}{n-j} &= \frac{\sgrboth{n-1}{i}}{n-1} \left[ \left( \frac{2}{n-j+2} - \frac{2}{n+1} \right) + \frac{n-j}{n-j+2} \right] \\
        &= \frac{\sgrboth{n-1}{i}}{n-1} \left[ \frac{n-j+2}{n-j+2} - \frac{2}{n+1} \right] = \frac{\sgrboth{n-1}{i}}{n-1} \frac{n-1}{n+1} = \frac{2i(n-i)}{n(n+1)^2}.
    \end{align*}
    Finally, adding $\sgrright{n-1}{i}g_n^j(n)$ we see that the right-hand side of the bulk recurrence simplifies to
    \begin{align*}
        \frac{i(i+1)}{n(n+1)}   \frac{2}{(n+1)(n+2)}  &+ \frac{2i(n-i)}{n(n+1)^2}
        = \frac{2i}{n(n+1)^2(n+2)} \Big[ i + 1 + (n-i)(n+2) \Big] \\
        &= \frac{2i}{n(n+1)^2(n+2)} \Big[ n^2 + 2n + 1 - i(n+1) \Big] \\
        &= \frac{2i(n+1)(n-i+1)}{n(n+1)^2(n+2)} = \frac{2i(n-i+1)}{n(n+1)(n+2)} = g_n^j(i),
    \end{align*}
    as desired.
    Since $(g_n^j(i))_{i,j\in C_n}$ satisfies both recurrences, and the solution to the recurrences is unique, we obtain $\mathsf{h}_n^j(i) = g_n^j(i)$ for all $n\in\N$ and all $i,j\in C_n$, concluding the proof.
\end{proof}
\begin{corollary}[Transition Probabilities to Vacant Configurations]\label{cor:hole-transitions}
    Let $\trans{n}$ denote the transition matrix of the stochastic sandpile Markov chain on the recurrent configurations of $C_n$ for $n\in\N$. The probabilities of transitioning to the vacant configurations are given by
    \begin{align*}
        &\trans{n}(\one_{C_n}, \one_{C_n} - \delta_i) = \frac{1}{3n} \quad &&\text{for all } i \in C_n,  \\
        &\trans{n}(\one_{C_n} - \delta_j, \one_{C_n} - \delta_i) = \frac{1}{3n} \quad &&\text{for } i \neq j,  \\
        &\trans{n}(\one_{C_n} - \delta_i, \one_{C_n} - \delta_i) = \frac{n-1}{3n} &&\text{for all }i\in C_n.
    \end{align*}
\end{corollary}

\begin{proof}
    We will go through the three different transition probabilities of Corollary \ref{cor:hole-transitions}.
    
    Firstly, observe for $i\in C_n$ that the Markov chain transitions from the sandpile $\one_{C_n}$ to $\one_{C_n} - \delta_i$, if we add a particle uniformly at random at some site $x \in C_n$ and the subsequent stabilization forms a hole at $i$. That is, 
    $$ \trans{n}(\one_{C_n}, \one_{C_n} - \delta_y) = \frac{1}{n} \sum_{x=1}^n \mathsf{h}_n^i(x) = \frac{1}{n} \sum_{x=1}^n \frac{\sgrboth{n}{x}}{n},$$
    where the second equality follows from Proposition \ref{prop:hole-uniformity}.
    Recall from the proof of Lemma \ref{lemma:transition-to-full} that $\sum_{x=1}^n \sgrboth{n}{x} = n/3$. Thus, the transition probability simplifies to $$ \trans{n}(\one_{C_n}, \one_{C_n} - \delta_i)=\frac{1}{n^2}  \frac{n}{3}  = \frac{1}{3n},$$
    as desired.
    Secondly, suppose we start at $\one_{C_n} - \delta_j$ and transition to a hole at $i < j$ for $i,j\in C_n$. This requires adding a particle at some $x < j$, and that the sub-segment $\{1, \dots, j-1\}$ loses two particles during stabilization to form a new hole at $i$. That is, 
    $$ \trans{n}(\one_{C_n} - \delta_j, \one_{C_n} - \delta_i) = \frac{1}{n} \sum_{x=1}^{j-1} \mathsf{h}_{j-1}^i(x) = \frac{1}{n} \sum_{x=1}^{j-1} \frac{\sgrboth{j-1}{x}}{j-1} = \frac{1}{n(j-1)} \frac{j-1}{3} = \frac{1}{3n}, $$
    where we again used Proposition \ref{prop:hole-uniformity} and the identity $\sum_{x=1}^n \sgrboth{n}{x} = n/3$.
    The case $i > j$ can be treated analogously.
    
    Finally, for all $i\in C_n$, the hole remains at $i$ if we add a particle at $x < i$ and the left sub-segment sends one particle to the left, or if we add a particle at $x > i$ and the right sub-segment exits one particle to the right. Summing these contributions yields
    $$ \trans{n}(\one_{C_n} - \delta_i, \one_{C_n} - \delta_i) = \frac{1}{n} \left( \sum_{x=1}^{i-1} \sgrleft{i-1}{x} + \sum_{x=i+1}^n \sgrright{n-i}{x-i} \right). $$
    Using the identity $\sum_{x=1}^m \sgrleft{m}{x} = \frac{m}{3}$, this evaluates to
    $$ \trans{n}(\one_{C_n} - \delta_i, \one_{C_n} - \delta_i)= \frac{1}{n} \left( \frac{i-1}{3} + \frac{n-i}{3} \right) = \frac{n-1}{3n}, $$
    which completes the proof.
\end{proof}
\subsection{Stationary Distribution on Vacant Configurations}\label{sec:vacant}
In this section, we will use the results on hole probabilities from Section \ref{sec:holes} to calculate the stationary distribution on the vacant configurations. That is, we will prove \ref{USS2} from Theorem \ref{thm:main}. Let us again consider the graph $C_n$ for some $n\in\N$.

We have already found the transition probabilities to the vacant configurations in Corollary \ref{cor:hole-transitions}, we can now use these results for the stationary probabilities.
\begin{proposition}[Stationary Probabilities of Vacant Configurations ]\label{prop:vacant-stat}
    Let $n\in\N$ and consider the stationary distribution $\statn{n}$ of the stochastic sandpile Markov chain on $C_n$. For all $v\in C_n$ it holds that
    $$\statn{n}(\one_{C_n}-\delta_v)=\frac{1}{2(n+1)},$$
    that is, the stationary distribution is uniform when restricted to the vacant configurations.
\end{proposition}
\begin{proof}
    Let $n\in\N$ and $v\in C_n$. Similarly  to the proof of Proposition \ref{prop:concentration-on-full} we use the definition of stationarity to obtain
    \begin{equation}
    \label{eq:stat-vacant}
    \begin{aligned}
        \statn{n}(\one_{C_n}-\delta_v)&=\statn{n}(\one_{C_n})\trans{n}(\one_{C_n},\one_{C_n}-\delta_v)\\&+\sum_{i=1}^n\statn{n}(\one_{C_n}-\delta_i)\trans{n}(\one_{C_n}-\delta_i,\one_{C_n}-\delta_v).
    \end{aligned}
    \end{equation}
    From Corollary \ref{cor:hole-transitions} together with Proposition \ref{prop:concentration-on-full} we obtain
    $$\statn{n}(\one_{C_n})\trans{n}(\one_{C_n},\one_{C_n}-\delta_v)=\bigg(\frac{1}{2}+\frac{1}{2(n+1)}\bigg)\frac{1}{3n}.$$
    Using Corollary \ref{cor:hole-transitions} we also obtain
    \begin{align*}
    \sum_{\substack{i=1\\ i\neq v}}^{v-1}\statn{n}(\one_{C_n}-\delta_i)\trans{n}(\one_{C_n}-\delta_i,\one_{C_n}-\delta_v)&=\frac{1}{3n}\sum_{\substack{i=1\\ i\neq v}}^n\statn{n}(\one_{C_n}-\delta_i)\\&=\frac{1}{3n}\bigg(\frac{1}{2}-\frac{1}{2(n+1)}-\statn{n}(\one_{C_n}-\delta_v)\bigg),
    \end{align*}
    as well as
    $$\statn{n}(\one_{C_n}-\delta_v)\trans{n}(\one_{C_n}-\delta_v,\one_{C_n}-\delta_v)=\frac{n-1}{3n}\statn{n}(\one_{C_n}-\delta_v).$$
    Reinserting this in Equation (\ref{eq:stat-vacant}) we obtain the equation
    \begin{align*}
        \statn{n}(\one_{C_n}-\delta_v)=\frac{1}{3n}+\frac{n-2}{3n}\statn{n}(\one_{C_n}-\delta_v),
    \end{align*}
    which we can solve for $\statn{n}(\one_{C_n}-\delta_v)$ to obtain
    $$(2n+2)\statn{n}(\one_{C_n}-\delta_v)=1\Rightarrow\statn{n}(\one_{C_n}-\delta_v)=\frac{1}{2(n+1)},$$
    which concludes the proof.
\end{proof}
Since all sets of the form $[[a,b]]$ for $a,b\in\Z$ with $a<0<b$ are simply translated versions of the set $C_{|a|+b+1}$, this also shows \ref{USS2} from Theorem \ref{thm:main}.

\subsection{Proof of the Infinite Volume Limit}\label{sec:proof-infinite-volume}

In this section, we are finally ready to prove the existence of the infinite volume limit as stated in \ref{USS3}. For the prove, we will use the results on the stationary distribution from \ref{USS1} and \ref{USS2}.

\begin{proof}[Proof of \ref{USS3} from Theorem \ref{thm:main}]
    Consider arbitrary sequences $(a_n)_{n\in\N}$ and $(b_n)_{n\in\N}$ with
    \begin{align*}
    \forall n\in\N: a_n<0<b_n,&&\text{and}&&\lim_{n\rightarrow\infty}a_n=-\infty,&&\text{and}&&\lim_{n\rightarrow\infty}b_n=\infty.
    \end{align*}
    Furthermore, let $A\subseteq\Z$ be any finite subset of the integers. Let us now choose some $n_0\in\N$ such that
    $$A\subseteq [[a_n,b_n]],$$
    for all $n>n_0$, which we can do since the set $A$ is assumed to be finite. We now see for $n>n_0$
    \begin{align*}
        \stat{a_n}{b_n}(\eta|_A\neq\one_A)=\sum_{i\in A}\stat{a_n}{b_n}(\one_{[[a_n,b_n]]}-\delta_i)=\frac{|A|}{2(|a_n|+b_n+2)}.
    \end{align*}
    Sending $n$ to infinity we thus obtain
    $$\lim_{n\rightarrow\infty}\stat{a_n}{b_n}(\eta|_A=\one_A)=1-\lim_{n\rightarrow\infty}\frac{|A|}{2(|a_n|+b_n+2)}=1,$$
    which shows
    $$\lim_{[[a,b]]\uparrow \Z}\stat{a}{b}=\delta_{\mathbbm{1}_\Z},$$
    where the limit is in the uniform weak sense.
\end{proof}
\section{Conclusion}\label{sec:conclusion}
In this paper, we were able to fully calculate the stationary distribution of stochastic sandpiles with the uniform toppling rule on finite line segments, as well as establish the existence of the infinite volume limit on the integers. We want to conclude this work, by addressing some interesting questions suitable for further research efforts.

\textbf{Monotonicity of the Density.} We have shown that the stationary distribution for the uniform toppling rule concentrates around half of its total mass on the full configuration on finite line segments. It seems unlikely that a similar concentration phenomenon occurs on more complicated graphs, however, we believe that the average density of the stationary distribution will always be strictly larger in the uniform toppling case.

\begin{figure}
    \centering
    \includegraphics[width=0.9\linewidth]{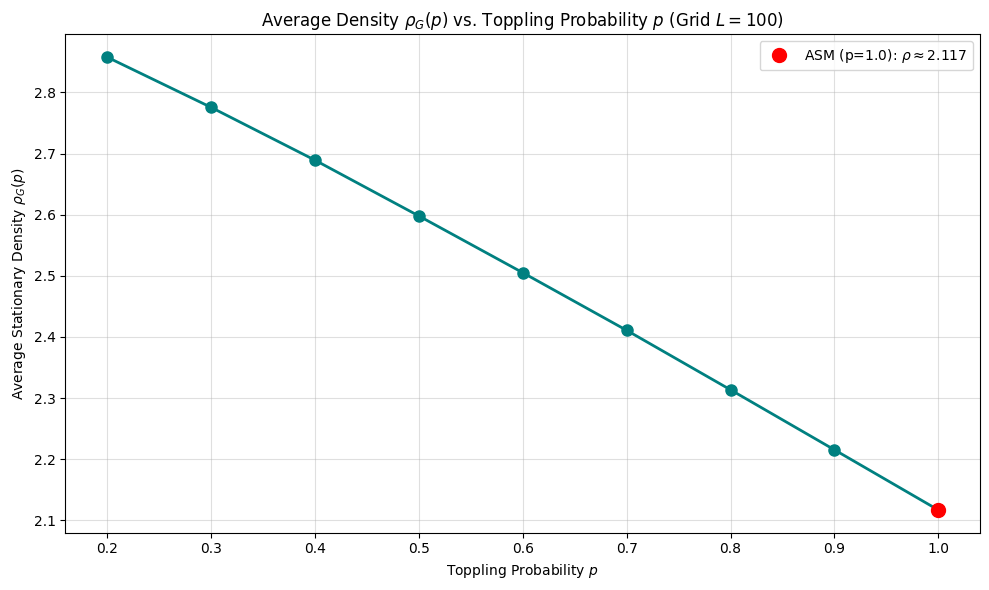}
    \caption{The graph illustrates    the average density $\rho_{G}(p)$ for $p$-topplings at certain values of $p$, where $G=[[1,100]]^2$ is a box with $10000$ vertices. We can see that the density seems to be linearly decaying as we increase the value of $p$.}
    \label{fig:monotonicity}
\end{figure}

Consider a finite graph $G=(V\cup\{s\},E)$ with designated sink vertex $s$. We consider $p$-topplings on $G$, where upon toppling a vertex $v\in V$, for each edge adjacent to $v$, we decide independently of the other edges to sent out a particle via the edge with probability $p$. Should a particle land at $s$, it is removed from the system. Notice that for $p=1$ we obtain the abelian sandpile model, and for $p=1/2$ all possible topplings are equally likely.

Let now denote by $\mu_G^{(p)}$ the stationary distribution of the stochastic sandpile Markov chain with the $p$-toppling rule. Let us now define the average density as
$$\rho_G(p):=\mathbb{E}_p\bigg[\sum_{v\in V}\frac{\eta(v)}{|V|}\bigg],$$
where $\mathbb{E}_p$ denotes the expectation with respect to the stationary measure $\mu_G^{(p)}.$ Our first question concerns the behaviour of the average density $\rho_G$.
\begin{conj}[Monotonicity of the Average Density]
    For all $p_1,p_2\in[0,1]$ with $p_1<p_2$ it holds that
    $$\rho_G(p_1)>\rho_G(p_2),$$
    that is, the average density $\rho_G$ is strictly monotonically decreasing in $p$.
\end{conj}
In fact, it would already be interesting to show, that the abelian sandpile model has the smallest average density.
\begin{conj}[Minimum Attained by Abelian Sandpiles]
    For all $p\in[0,1)$ it holds that
    $$\rho_G(p)>\rho_G(1).$$
\end{conj}

In Figure \ref{fig:monotonicity} we have simulated $p$-topplings on a finite two-dimensional box, and our results further seem to verify the conjectured monotonicity.

\textbf{Percolation of Height 3 Vertices.} We have also simulated how a typical sample of the stationary distribution for $p$-topplings looks on finite boxes in the two-dimensional grid. We refer to Figure \ref{fig:height-3-perc-samples} for an illustration of such samples for different values of $p$.

\begin{figure}
    \centering
    \includegraphics[width=\linewidth]{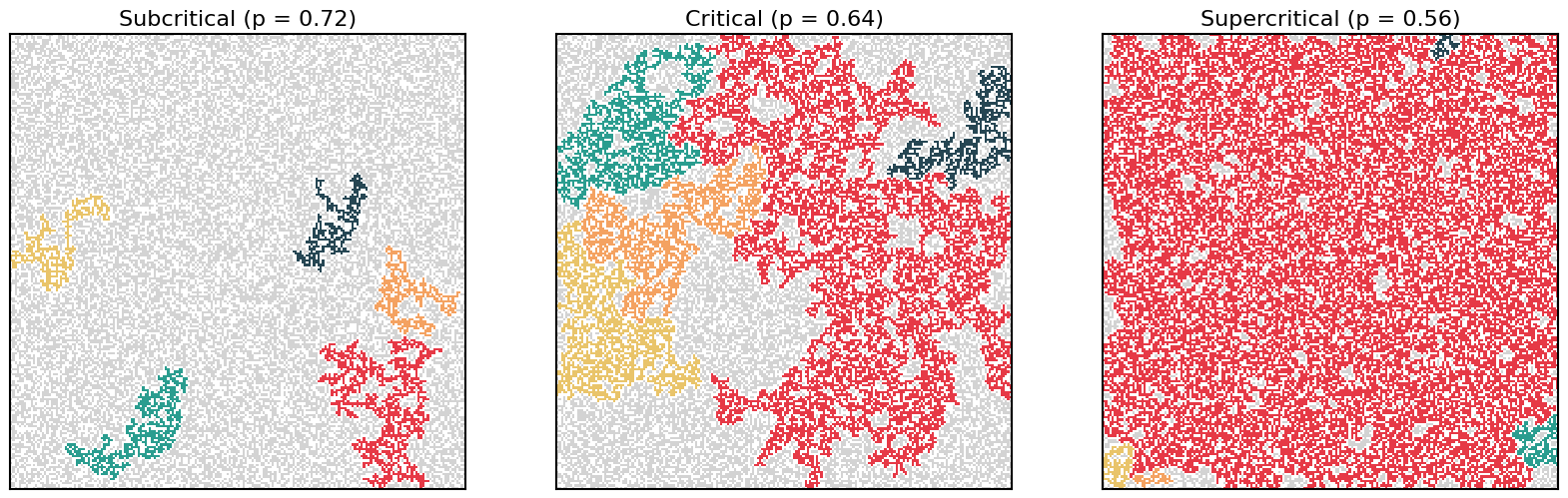}
    \caption{The figure depicts the height $3$ cluster of typical samples of the stationary distribution of stochastic sandpiles with $p$-topplings on a box of side length $200$ for the subcritical and supercritical regime, as well as for a value of $p$ close to the critical regime. The five largest clusters are highlighted in color, the remaining clusters are depicted in gray.}
    \label{fig:height-3-perc-samples}
\end{figure}
Consider for $n\in\N$ the stationary distribution with $p$-topplings $\mu_n^{(p)}$ on a square box $[[1,n]]^2$, where particles are lost when sent outside of the box. Furthermore, denote for a given sandpile configuration $\eta$ the set of height $3$ vertices by
$$T_n(\eta):=\{v\in[[1,n]]^2:\eta(v)=3\}.$$
Our simulations suggest that, as we make $p$ smaller, the typical size of the set $T_n(\eta)$ increases. In fact, for values of $p$ that are close to $1$, there is only a sporadic number of height $3$ vertices. However, as we decrease $p$, eventually the set $T_n(\eta)$ starts to connect the left boundary of the box with the right boundary, that is, it percolates.

\begin{conj}
    There exists $p_c\in (0,1)$ such that for all $p<p_c$ we have that
    $$\mu_n^{(p)}(\text{There exists a path in }T_n(\eta)\text{ that connects the left and right boundary})\xrightarrow{n\rightarrow\infty}1,$$
    and for all $p>p_c$
    $$\mu_n^{(p)}(\text{There exists a path in }T_n(\eta)\text{ that connects the left and right boundary})\xrightarrow{n\rightarrow\infty}0.$$
\end{conj}

In Figure \ref{fig:height-3-perc-phase-transition}, we have plotted the percolation probability against the value of $p$, and it seems that there is a sharp decline around the value $p\approx 0.64.$

\begin{figure}
    \centering
    \includegraphics[width=\linewidth]{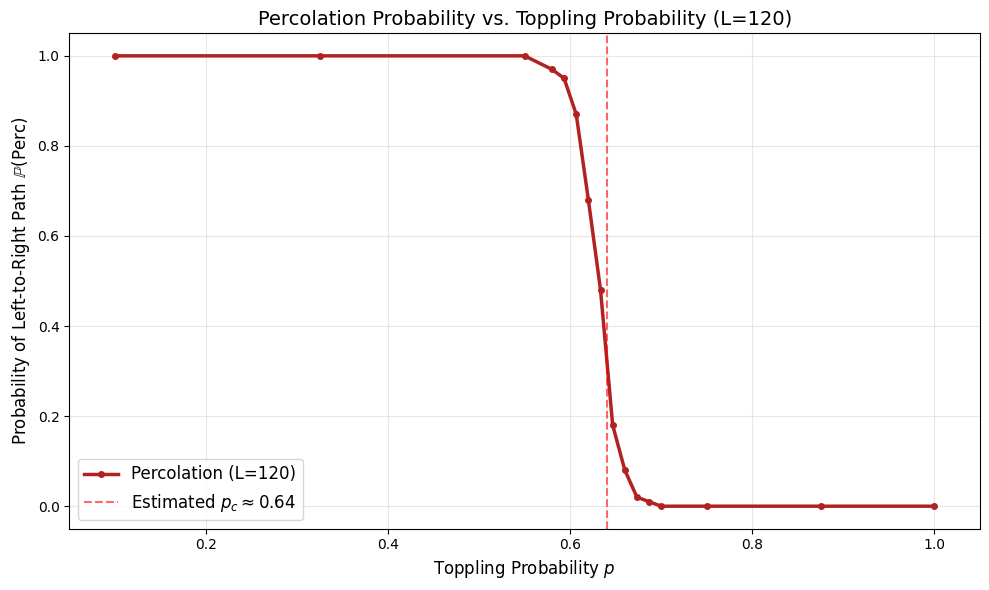}
    \caption{The graph shows the probability of finding a cluster of height $3$ vertices spanning from the right side of a box of side length $L=120$ to the left side of the box. Our simulations suggest that the phase transition occurs when the average density of height $3$ vertices is equal to the critical Bernoulli site percolation threshhold on $\mathbbm{Z}^2$.}
    \label{fig:height-3-perc-phase-transition}
\end{figure}
\newpage
\textbf{Single-Source Sandpiles.} Consider stochastic sandpiles with uniform topplings on the integers $\Z$. Our final question concerns the shape we obtain when we place $n\in\N$ vertices at the origin, and then stabilize the configuration. Our simulations strongly suggest the following conjecture. 

\begin{conj}
    Let us denote by $D_n$ the set of toppled vertices in the stabilization of $(n\delta_0)$ on the integers. Then
    $$\prob\bigg([-(1-\varepsilon)\frac{n}{2},(1-\varepsilon)\frac{n}{2}]\subseteq D_n \subseteq [-(1+\varepsilon)\frac{n}{2},(1+\varepsilon)\frac{n}{2}]\bigg)\xrightarrow{n\rightarrow\infty}1,$$
    for all $\varepsilon>0.$
\end{conj}
In fact, it would be interesting to also consider the limit shape of the single-source model on other graphs, such as $\Z^d$ for $d\geq 2$, or for different toppling rules, such as the $p$-topplings as previously described.  

\textbf{Funding Information.} This research was funded in part by the Austrian Science Fund (FWF) 10.55776/P34713.

\bibliographystyle{alpha}
\bibliography{lit}

\textsc{David Beck-Tiefenbach}, Universität Innsbruck, Institut für Mathematik, Technikerstraße 23, A-6020 Innsbruck, Austria. \texttt{david.beck-tiefenbach@uibk.ac.at}

\textsc{Robin Kaiser}, Departement of Mathematics, CIT, Technische Universität München, Boltzmannstr. 3, D-85748 Garching bei München, Germany. \texttt{ro.kaiser@tum.de}
\end{document}